\documentclass[11pt]{article}

\usepackage{amsfonts}
\usepackage{amssymb}
\usepackage{amsmath,mathrsfs}
\usepackage{amsthm}
\usepackage{latexsym}
\usepackage{layout}

\usepackage{textcomp}
\usepackage{amsbsy}
\usepackage{latexsym}
\usepackage[mathscr]{eucal}
\usepackage{amsfonts}
\usepackage{amssymb}
\usepackage[usenames]{color}
  \usepackage{graphics} 
  \usepackage{epsfig} 
\usepackage{graphicx}  \usepackage{epstopdf}

\usepackage[english]{babel}

\theoremstyle{plain}
\begingroup
\newtheorem{theorem}{Theorem}[section]
\newtheorem{lemma}[theorem]{Lemma}
\newtheorem{proposition}[theorem]{Proposition}
\newtheorem{corollary}[theorem]{Corollary}
\endgroup

\theoremstyle{definition}
\begingroup
\newtheorem{definition}[theorem]{Definition}
\newtheorem{remark}[theorem]{Remark}

\endgroup
 
\theoremstyle{remark}
\begingroup

\endgroup

\mathchardef\emptyset="001F

\numberwithin{equation}{section}

\newcommand{\R}{{\mathbb R}}
\newcommand{\PP}{{\mathcal P}_1}
\newcommand{\WW}{{\mathcal W}_1}
\newcommand{\U}{{\mathcal U}}

\newcommand{\supp}{\mathrm{supp}}

\begin{document}

\title{Mean-Field Sparse Optimal Control}


\author{Massimo Fornasier\thanks{Technische Universit\"at M\"unchen, Fakult\"at Mathematik, Boltzmannstrasse 3
 D-85748, Garching bei M\"unchen, Germany. {\tt massimo.fornasier@ma.tum.de}}, Benedetto Piccoli\thanks{Department of Mathematical Sciences, Rutgers University - Camden, Camden, NJ. {\tt piccoli@camden.rutgers.edu}}, Francesco Rossi\thanks{Aix-Marseille Univ, LSIS, 13013, Marseille, France. {\tt francesco.rossi@lsis.org}}}

\maketitle

\begin{abstract}
We introduce  the rigorous limit process connecting finite dimensional sparse optimal control problems with ODE constraints, modeling parsimonious interventions on the dynamics of a moving population
divided into leaders and followers, to an 
infinite dimensional optimal control problem with a constraint given by a system of ODE for the leaders
coupled with a PDE of Vlasov-type, governing the dynamics of the probability distribution of the followers. 
In the classical mean-field theory one studies the behavior of a large number of small individuals {\it freely interacting} with each other, by 
simplifying the effect of all the other individuals on any given individual by a single averaged effect. In this paper we address instead 
the situation where the leaders are actually influenced also by an external {\it policy maker}, and we propagate its effect for the number $N$ of followers going to infinity. 
The technical derivation of the sparse mean-field optimal control is realized by the simultaneous development of the mean-field limit
of the equations governing the followers dynamics together with the $\Gamma$-limit of the {finite dimensional} sparse optimal control problems.
\end{abstract}

\vspace{.3cm}
{\bf Keywords: }  Sparse optimal control, mean-field limit, $\Gamma$-limit,  optimal control with ODE-PDE constraints.\\


\vspace{1cm}
\tableofcontents
\section{Introduction}

In several individual based models for multi-agent motion the finite-dimensional dynamics in $2 d \times N$ variables, where $N$ is the number of individuals and $d$ is the dimension of the space in which {the motion of} such individuals evolves, is given by
\begin{equation}
\begin{cases}
\dot x_i = v_i, & \\
\dot v_i =  H\star \mu_N(x_i,v_i), & i=1,\dots N, \quad t \in [0,T],
\end{cases}\label{startsys}
\end{equation}
where $H : \mathbb R^{2d} \to \mathbb R^d$ is a locally Lipschitz interaction kernel with sublinear growth whose action on the group is modeled by  convolution, {where  the atomic measure
\begin{equation}
\mu_N(t) = \frac{1}{N} \sum_{i=1}^N \delta_{(x_i(t),v_i(t))}
\label{muNfoll}
\end{equation}
differently represents  the group of agents.}  As a relevant example of this setting, we mention the interaction kernel $H(x,v):= a(|x|)v$, for a bounded nonincreasing function $a: \mathbb R_+ \to \mathbb R_+$, which gives the well-known alignment model of Cucker and Smale \cite{CS,cusm07}, see also the generalizations in \cite{HaHaKim}, as well as interaction kernels of the type $H(x,v):= f(|x|)x$, where the function $f: \mathbb R_+ \to \mathbb R$ can encode small range repulsion and medium-long range attraction, as considered in \cite{CuckerDong13}, see Figure \ref{fig:repatt}. 

\begin{figure}[!htb]
\centering
\includegraphics[scale=0.6]{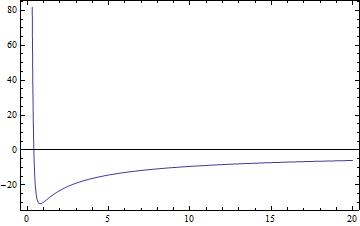}
\caption{Model function $f(r)= \frac{\sigma_r}{r^{4}} - \frac{\sigma_a}{r^{0.4}}$, {$\sigma_r>0$, $\sigma_a>0$,} of small range repulsion and medium-long range attraction.}
\label{fig:repatt}
\end{figure}
As discussed in details in the aforementioned papers, such systems can exhibit convergence  to certain interesting attractors, representing a higher level of global organization, although such spontaneous coordination may be conditional, depending
on the initial configuration.     In the recent work  \cite{bofo13,CFPT} the external control of such systems has been considered in order to promote the collective organization of the group of agents also in those situations where the initial conditions are out of the basin of attraction of the interesting configurations. The emphasis given in this context was on {\it sparse} controls, meaning that we consider systems 
\begin{equation}
\begin{cases}
\dot x_i = v_i, & \\
\dot v_i =  H\star \mu_N(x_i,v_i) + u_i, & i=1,\dots N, \quad t \in [0,T],
\end{cases}\label{contrsys}
\end{equation}
where $u_i:[0,T] \to \mathbb R^d$ are  measurable control functions which we wish being vanishing for most of the $i=1,\dots,N$ and possibly for most of the $t \in [0,T]$. This choice of controls models the {\it parsimonious} and {\it moderate}
external intervention of a government of the group, for instance the role of a mediator in an assembly, where the group needs to reach unanimous consensus on a common conduct, as it is the case for the voting system in the Council of the European Union, where  unanimous decision  are usually targeted.
\\

When the number of the involved agents $N$ is very large, the solution of an optimal control problem for a system of the type \eqref{contrsys} unfortunately becomes an impossible task because of the {\it curse of dimensionality}. Already dealing with systems of a few hundreds agents is computationally extremely demanding and often numerically inaccurate. Therefore, we may wonder whether we can describe an appropriate limit dynamics and an optimal control problem for the limit case $N \to +\infty$, which can be re-conducted to computationally manageable dimensionalities. When no control is involved, this procedure is well-known as in the classical mean-field theory one studies the evolution of a large number of small individuals {\it freely interacting} with each other, by 
simplifying the effect of all the other individuals on any given individual by a single averaged effect. This results in considering the evolution of the particle density distribution in
the state variables, leading to so-called mean-field partial differential equations of Vlasov- or Boltzmann-type \cite{1151.82351}. In particular, for our system \eqref{startsys} the corresponding mean-field equations are 
$$
\partial_t\mu+v \cdot \nabla_x \mu = \nabla_v \cdot \left [ \left ( H \star \mu \right ) \mu \right ]. 
$$
We refer to \cite{CCH13} and the references therein 
for a recent survey  on some of the most relevant mathematical aspects on this approach to swarm models.
Nevertheless, the proper definition of a limit dynamics when an external control is added to the system and it is supposed to have some {\it sparsity} surprisingly remains a difficult task. In fact, the most immediate and perhaps natural approach would be to assign as well to the finite dimensional control $u$ an atomic vector valued time-dependent measure
$$
\nu_N(t) = \sum_{i=1}^N u_i \delta_{(x_i(t),v_i(t))}
$$
and consider a proper limit $\nu$ for $N \to +\infty$, leading to the controlled PDE
\begin{equation}\label{nufkup}
\partial_t\mu+v \cdot \nabla_x \mu = \nabla_v \cdot \left [ \left ( H \star \mu  \right ) \mu + \nu \right ], 
\end{equation}
where now $\nu$ represents an external source field. Unfortunately, despite the fact that $\nu$ is supposed to be the minimizer of certain cost functionals which may allow for the necessary compactness to derive the limit $\nu_N \to \nu$, 
it seems eventually {hard} to design a cost functional with a proper meaning in the finite dimensional model and at the same time promoting a good behavior of the measure $\nu$. In fact, for the optimal control problems considered for instance in \cite[Section 5]{CFPT} such a limit procedure does not prevent $\nu$ to be singular with respect to $\mu$. This means that in the weak formulation of the equation \eqref{nufkup} the role of $\nu$ is essentially mute, it does not interact at all with $\mu$, hence it loses completely its steering purpose. Imaginatively, it is like trying to steer a river by means of toothpicks! Even if we considered in \eqref{nufkup}  the absolutely continuous part only $\mu_a = f \mu$ of $\nu$ with respect to $\mu$, if there was any, we would end up with an equation of the type 
\begin{equation}\label{nufkup2}
\partial_t\mu+v \cdot \nabla_x \mu = \nabla_v \cdot \left [ \left ( H \star \mu  + f \right ) \mu \right ], 
\end{equation}
where now $f$ is a force field which is just an $L^1$-function with respect to the measure $\mu$. Unfortunately, existence and stability of solutions for equations of the type \eqref{nufkup2} is established only for fields $f$ with at least {\it some} regularity \cite{am08}.
At this point it seems that our quest for a proper definition of a {\it mean-field optimal control} gets to a dead-end, unless we allow for some modeling compromise. The first successful approach actually starts from the equation \eqref{nufkup2}, by assuming $f(t,x,v)$ being in a proper compact set of a function space of  Carath{\'e}odory functions in $t$ and locally Lipschitz continuous functions in $(x,v)$, and proceeding back to reformulate the finite dimensional modeling, leading to systems of the type
\begin{equation}
\begin{cases}
\dot x_i = v_i, & \\
\dot v_i =  H\star \mu_N(x_i,v_i) + f(t,x_i,v_i), & i=1,\dots N, \quad t \in [0,T],
\end{cases}\label{contrsys2}
\end{equation}
where now $f$ is a feedback control. This approach has been recently explored in \cite{MFOC}, where a proof of a simultaneous $\Gamma$-limit and mean-field limit of the finite dimensional optimal controls for \eqref{contrsys2} to a corresponding infinite dimensional optimal control for \eqref{nufkup2} has been established. We also mention the related work \cite{befrph13} where first order conditions are derived for optimal control problems of equations of the type \eqref{nufkup2} for Lipschitz feedback controls $f(t,x,v)$ in a stochastic setting. Such conditions result in a coupled system of a forward Vlasov-type equation and a backward Hamilton-Jacobi equation, {similarly to situations encountered in the context of {\it mean-field games} \cite{lali07} or the {\it Nash certainty equivalence}  \cite{HCM03}.}
 Certainly, this calls for a renewed enthusiasm and hope, until one realizes that actually the problem of characterizing the optimal controls $f(t,x,v)$ with the purpose of an efficient and manageable numerical computation may not have simplified significantly, as it is not a trivial task to obtain a rigorous derivation and the well-posedness of the corresponding first order conditions as in  \cite{befrph13} in a fully deterministic setting. This introduces us to the main scope of this paper. Inspired by the successful construction of the coupled $\Gamma-$ and mean-field-limits in \cite{MFOC}
and the multiscale approach in \cite{crpito10-1,crpito11},  {to describe a mixed granular-diffuse dynamics of a crowd}, we modify here our modeling not starting anymore from \eqref{nufkup2}, but actually from the initial system \eqref{startsys}.
\\
Let us now add to \eqref{startsys} $m$ particular individuals, which interact freely with the $N$ individuals given above. We denote by $(y,w)$ the space-velocity variables of these new individuals. We shall consider these $m$ individuals as ``leaders'' of the crowd, while the other $N$ individuals will be called  ``followers''. We assume that we have a small amount $m$ of leaders that have a great influence on the population, and a large amount $N$ of followers which have a small influence on the population.

Then, the dynamics we shall study is
\begin{equation}
\begin{cases}
\dot y_k = w_k, & \\
\dot w_k =  H\star \mu_N(y_k,w_k) + H \star\mu_m(y_k,w_k)\, & k=1,\dots m, \quad t \in [0,T],\\
\dot x_i = v_i, & \\
\dot v_i =  H\star \mu_N(x_i,v_i) + H \star\mu_m(x_i,v_i)\, & i=1,\dots N, \quad t \in [0,T],
\end{cases}
\end{equation}
where we considered the additional atomic measure
\begin{equation}\label{atomicleaders}
\mu_m(t) = \frac{1}{m} \sum_{k=1}^n \delta_{(y_k(t),w_k(t))},
\end{equation}
supported on the trajectories $t \mapsto (y_k(t),w_k(t))$, $k=1,\dots,m$. (One can generalize this model to the one where different kernels for the interaction between a leader and a follower, two leaders, etc. are considered. All the results of this paper easily generalize to this setting.) From now on, the notations $\mu_N$ and $\mu_m$ for the atomic measures representing followers and leaders respectively will be considered fixed and we shall  use them extensively in the rest of the paper.
Up to now, the dynamics of the system is similar to a standard multi-agent dynamics for $N+m$ individuals, with the only difference that the actions of leaders and followers have different weights on a single individuals, $\frac{1}{m}$ and $\frac{1}{N}$, respectively. Let us now add controls on the $m$ leaders. We obtain the system
\begin{equation}
\begin{cases}
\dot y_k = w_k, & \\
\dot w_k =  H\star \mu_N(y_k,w_k) + H \star\mu_m(y_k,w_k) + u_k\, & k=1,\dots m, \quad t \in [0,T],\\
\dot x_i = v_i, & \\
\dot v_i =  H\star \mu_N(x_i,v_i) + H \star\mu_m(x_i,v_i)\, & i=1,\dots N, \quad t \in [0,T],
\end{cases}
\label{e-findim}
\end{equation}
where $u_k:[0,T]\to \R^d$, are measurable controls {for $k=1,\dots,m$, and we define  the control map  $u:[0,T]\to \R^{md}$ by $u(t)=(u_1(t),\dots,u_m(t))$ for $t \in [0,T]$.}  The main difficulty arising in this context is that one usually deals with control functions $u(\cdot)$ that are discontinuous in time. In fact, one needs to consider solutions of the finite-dimensional problem \eqref{e-findim} in the Carath{\'e}odory sense, i.e., functions $t\mapsto(y(t),w(t),x(t),v(t))$ that are absolutely continuous with respect to time and satisfy the integral formulation of \eqref{e-findim}. For the sake of completeness and readability of our results we report some well-known facts on such solutions in the Appendix. In this setting, it makes sense to choose
$u\in L^1([0,T],\U)$ where $\U$ is a fixed nonempty compact subset of $\mathbb R^{d \times m}$ and $\U \subset B(0,U)$ for $U>0$.  Finite-dimensional control problems in this setting are of interest, and we will focus on a specific class of control problems, namely optimal control problems in a finite-time horizon with fixed final time. We design  {\it sparse} control $u$ to drive the whole population of $m+N$ individuals to a given configuration. We model this situation by solving the following optimization problem 
\begin{equation}\label{sparseoptcontr}
\min_{u \in L^1([0,T],\U)} \int_{0}^T  \left \{ L(y(t),w(t),\mu_N(t))   +  \frac{1}{m} \sum_{k=1}^m |u_k(t)|  \right \}dt,
\end{equation}
where $L(\cdot)$ is a suitable continuous map in its arguments. (For example, one can use $L$ to model the distance between the state variables and the basin of attraction to the interesting configurations. Then the optimization leads the system to goal-driven dynamics.) The use of (scalar) $\ell^1$-norms to penalize controls as in \eqref{sparseoptcontr} dates back to the 60's with the models of linear fuel consumption \cite{crlo65}. 
More recent work in dynamical systems \cite{voma06} resumes again $\ell^1$-minimization emphasizing its sparsifying power. 
Also in optimal control with partial differential equation constraints it became rather popular to use $L^1$-minimization to enforce sparsity of controls \cite{caclku12,clku11,clku12,hestwa12,pive12,st09,wawa11}, for instance
in the modeling of optimal placing of actuators or sensors.
\\

In order to give precise meaning to the limit of the optimal control problems \eqref{e-findim}-\eqref{sparseoptcontr} for the number $N$ of followers tending to infinity, we need to address a few technical challenges.
As already observed above, due to the presence of the control $u(\cdot)$, the classical results for the mean-field limit of \eqref{e-findim} cannot be directly applied,  because here the right-hand side is discontinuous in time, see for instance \cite{ambrosio,CanCarRos10,gw2,pedestrian} where continuity of the right-hand-side is assumed. Moreover, only a part of the $m+N$ variables increases in number, while the number $m$ of leaders is kept constant. Finally, even a description of the whole population of leaders and followers by a unique measure would not catch the possibility of acting on the leaders only. \\

As one of our main results, we shall show in Theorem \ref{p-infdimsol} that, given a control strategy $u\in L^{1}([0,1],\U)$, it is possible to formally define a mean-field limit of \eqref{e-findim} when $N\to\infty$ in the following sense: the population is represented by the vector of positions-velocities $(y,w)$ of the leaders coupled with the compactly supported probability measure $\mu\in \PP(\R^{2d})$ of the followers in the position-velocity space. Then, the mean-field limit will result in a coupled system of an ODE with control for $(y,w)$ and a PDE without control for $\mu$. More precisely the limit dynamics will be described by
\begin{equation}
\begin{cases}
\dot y_k = w_k, & \\
\dot w_k =  H\star(\mu+\mu_m)(y_k,w_k)+u_k, & k=1,\dots m, \quad t\in[0,T]\\
\partial_t\mu+v \cdot \nabla_x \mu = \nabla_v \cdot \left [ \left ( H \star (\mu +\mu_m) \right ) \mu \right ],
\end{cases}
\label{e-infdim0}
\end{equation}
where the weak solutions of the equations have to be interpreted in the Carath{\'e}odory sense. See Figure \ref{fig:multiscale} for an example of the dynamics of \eqref{e-infdim0} for a multiscale pedestrian crowd mixing a granular discrete part and a diffuse part.

\begin{figure}[!htb]
\centering
\includegraphics[scale=0.7]{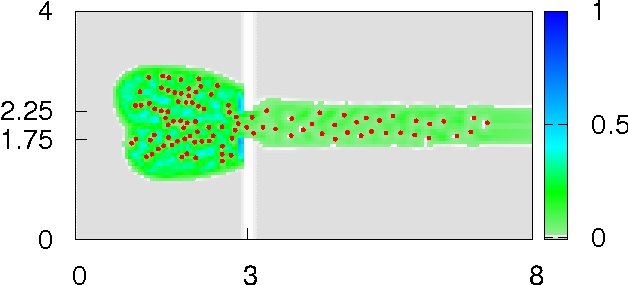}
\caption{A mixed granular-diffuse crowd leaving a room through a door. This figure was kindly provided by the authors of \cite{crpito11}. Copyright \textcopyright2011 Society for Industrial and Applied Mathematics.  Reprinted with permission.  All rights reserved.}
\label{fig:multiscale}
\end{figure}

Simultaneously, we shall prove in Theorem \ref{thm:gamma} a $\Gamma$-convergence result,  implying that the optimal controls $u^*_N$ of the finite dimensional optimal control problems \eqref{e-findim}-\eqref{sparseoptcontr} converge weakly in  $L^1([0,T],\U)$ for $N \to +\infty$ to optimal controls $u^*$, which are minimal solutions of 
\begin{equation}
\min_{u \in L^1([0,T],\U)} \int_{0}^T  \left \{ L(y(t),w(t),\mu(t))   +  \frac{1}{m} \sum_{k=1}^m |u_k(t)|  \right \}dt.
\label{Ffun0}
\end{equation}
This is actually an existence result of solutions for the infinite-dimensional optimal control problem \eqref{e-infdim0}-\eqref{Ffun0}. Differently from the one proposed in \cite{MFOC} though, this model retains the controls only on a finite and small group of agents, despite the fact that the entire population can be very large (here modeled by the limit $N\to +\infty$). Hence, by the stratagem of dividing the populations in two groups and allowing only one of them to have growing size we do not need anymore to be necessarily exposed to the curse of dimensionality when it comes to numerically solving the corresponding optimal control problem. We shall address the concrete analysis of the first order optimality conditions for \eqref{e-findim}-\eqref{sparseoptcontr} and their relationship to \eqref{e-infdim0}-\eqref{Ffun0} in a follow-up paper. This will be the basis for the numerical implementations.
\\

The paper is organized as follows. In Section 2 {we apply basic results recalled from the Appendix to ensure} the well-posedness of the finite dimensional system \eqref{e-findim}. Section 3 will be devoted to the mean-field limit of \eqref{e-findim} to the coupled system \eqref{e-infdim0} and the well-posedness of the latter. For the sake of self-containedness we sketch in Section 4 known existence results for the finite dimensional problems \eqref{e-findim}-\eqref{sparseoptcontr}. In Section 5 we develop our main result of $\Gamma$-convergence of the finite dimensional optimal control problems \eqref{e-findim}-\eqref{sparseoptcontr} to the corresponding infinite dimensional ones \eqref{e-infdim0}-\eqref{Ffun0}.
{The concluding Appendix recalls classical well-posedness results of Carath{\'e}odory differential equations and certain stability results of transport flows specifically formulated for the systems of equations \eqref{e-findim} and \eqref{e-infdim0}.}
\section{The finite-dimensional dynamics}

We state the following assumptions:
\begin{itemize}
\item[(H)] Let $H\colon \R^{2d} \to \R^{d}$ be a locally Lipschitz function such that, for a constant $C>0$
\begin{equation}\label{lingrowth}
|H(\xi)|\le C(1+|\xi|), \quad \mbox{for all } \xi \in \mathbb R^{2d};
\end{equation}
\end{itemize}

We consider now the system \eqref{e-findim} with $N$ followers and  the control $u$. We shall prove results of existence and uniqueness of the solution of \eqref{e-findim}, where time-dependent support estimates will be given independently on the number $N$ of followers. With this goal, we endow each space of configurations $\R^{2d(m+N)}$ with the following norm and the corresponding distance:
\begin{equation}
\| (y,w,x,v)\|:=\frac{1}{m}\sum_{k=1}^m(|y_k|+|w_k|)+\frac{1}{N}\sum_{i=1}^N (|x_i|+|v_i|),
\label{e-dconf}
\end{equation}
where the norm $| \cdot|$ on $\R^{d}$ is the Euclidean. The choice of  this norm \eqref{e-dconf} will be eventually related to the use of the $1$-Wasserstein distance on the space $\PP(\R^{2d})$ of probability measures of bounded first moment. For the sake of a compact writing we shall denote
the trajectories of \eqref{e-findim} by $\zeta(t) = (y(t),w(t),x(t),v(t))$ and trajectories of leaders or followers by $\xi$, i.e., $\xi(t)=(y(t),w(t))$ or $\xi(t)=(x(t),v(t))$ depending on the context. We can write \eqref{e-findim} in the following compact form
\begin{equation}
 \dot{\zeta}(t) = g(t, \zeta(t)),
\label{compactform}
\end{equation}
where the right-hand side is
\begin{equation}\label{explRHS}
g(t,\zeta) = \Big (w, [H \star (\mu_N + \mu_m)(y_k,w_k) + u_k]_{k=1}^m, v ,  [H \star (\mu_N + \mu_m)(x_i,v_i)]_{i=1}^N \Big ).
\end{equation}
\begin{lemma}\label{lem:fd}
Given $H$ satisfying condition (H) and $\mu_n = \frac{1}{n} \sum_{\ell=1}^n \delta_{\xi_\ell}$ for $\xi_\ell \in \mathbb R^{2d}$ for all $\ell=1,\dots,n$, an arbitrary atomic measure, we have
\begin{equation}
| H \star \mu_n (\xi) | \leq C \left (1 + |\xi| + \frac{1}{n} \sum_{\ell=1}^n | \xi_n| \right ).
\end{equation}
\end{lemma}
\begin{proof}
By sublinear growth of $H$ we have immediately the estimate
\begin{eqnarray*}
| H \star \mu_n (\xi) | &\leq& \frac{1}{n} \sum_{\ell=1}^n |H(\xi - \xi_\ell)| \leq C( 1 + |\xi| + \frac{1}{n} \sum_{\ell=1}^n | \xi_n|).
\end{eqnarray*}
\end{proof}

\begin{proposition}\label{exsol}
Let $H$ be a map satisfying (H). Then given a control $u \in L^1([0,T],\U)$ and an initial datum $\zeta^0=(y^0,w^0,x^0,v^0)$ there exists a unique Carath{\'e}odory solution $\zeta(t)=(y(t),w(t),x(t),v(t))$ of \eqref{e-findim} such that
\begin{equation}\label{growth}
\| \zeta(t) \| \leq (\| \zeta^0\| + \bar C T) e^{\bar C T},
\end{equation}
for all $t \in [0,T]$, where $\bar C>0$ is a constant depending on $C>0$, $U>0$ but not depending on $N$. Moreover, the trajectory is Lipschitz continuous in time, i.e., 
\begin{equation}\label{lipcont}
\| \zeta(t_1) -  \zeta(t_2)\| \leq \mathcal L |t_1 - t_2|, \quad \mbox{ for all } t_1,t_2 \in [0,T],
\end{equation}
for the Lipschitz constant $\mathcal L = \bar C(1+(\| \zeta^0\| + \bar C T) e^{\bar C T})$. 
\end{proposition}
\begin{proof}
Given the explicit form of \eqref{explRHS} and thanks to condition (H) and Lemma \ref{lem:fd}, the right-hand side $g(t,\zeta)$ of the system \eqref{compactform} fulfills the linear growth condition
$$
\|g(s,\zeta)\| \leq \bar C (1 + \|\zeta\|), \quad \mbox{for all } \zeta \in \mathbb R^{2d},
$$
allowing us to apply Theorem \ref{cara-global} in the Appendix, which ensures the well-posedness of \eqref{e-findim}.
Moreover,
$$
\| \zeta(t_1) - \zeta(t_2) \| = \left \|  \int_{t_2}^{t_1} g(s,\zeta(s)) ds \right \| \leq  \int_{t_2}^{t_1} \bar C (1 + \|\zeta(s)\|) ds \leq \bar C(1+(\| \zeta^0\| + \bar C T) e^{\bar C T})  |t_1-t_2|.
$$
\end{proof}

\section{The coupled ODE and PDE system}

In the following we consider the space $\mathcal P_1(\mathbb R^n)$, consisting of all probability measures on $\mathbb R^n$ of finite
first moment. On this set we shall consider the following distance, called the {\it Monge-Kantorovich-Rubistein distance},
\begin{equation}\label{mkrdist}
\mathcal W_1(\mu,\nu)=\sup \left \{ \left | \int_{\mathbb R^n} \varphi(x) d (\mu-\nu)(x)  \right| : \varphi \in \operatorname{Lip}(\mathbb R^n), \quad \operatorname{Lip}(\varphi) \leq 1 \right \}, 
\end{equation}
where $\operatorname{Lip}(\mathbb R^n)$ is the space of Lipschitz continuous functions on $\mathbb R^n$ and $\operatorname{Lip}(\varphi)$ the Lipschitz constant of a function $\varphi$. 
Such a distance can also be represented in terms of optimal transport plans by Kantorovich duality in the following manner: if we denote $\Pi(\mu,\nu)$ the set of transference plans
between the probability measures $\mu$ and $\nu$, i.e., the set of probability measures on $\mathbb R^n \times \mathbb R^n$ with first and second marginals equal to $\mu$ and
$\nu$ respectively, then we have
\begin{equation}\label{wasserstein}
\mathcal W_1(\mu,\nu) =\inf_{\pi \in \Pi(\mu,\nu)} \left  \{ \int_{\mathbb  R^n \times \mathbb R^n} |x-y| d \pi(x,y) \right \}.
\end{equation}
In the form \eqref{wasserstein} the distance $\mathcal W_1$ is also known as the $1$-Wasserstein distance. 
We refer to \cite{AGS,vi09} for more details. Notice that if $\mu_m = \frac{1}{m} \sum_{k=1}^m \delta_{\xi_k}$ and  $\mu_m' = \frac{1}{m} \sum_{k=1}^m \delta_{\xi_k'}$ are two atomic measures, then \eqref{mkrdist} immediately yields
\begin{equation}\label{wasserest}
\mathcal W_1(\mu_m,\mu_m') \leq \frac{1}{m} \sum_{k=1}^m | \xi_k - \xi_k'|.
\end{equation}
This is the reason for having fixed the norm notation $\|\cdot\|$ as in \eqref{e-dconf}.
\\

We formally define now a proper concept of solutions for the system \eqref{e-infdim0}.

\begin{definition}\label{soldef}
 Let $u\in L^1([0,T],\U)$ be given. We say that a map $(y,w,\mu):[0,T]\to \mathcal X:=\R^{2d \times m}\times \PP(\R^{2d})$ is a solution of the  controlled system with interaction kernel $H$ 
\begin{equation}
\begin{cases}
\dot y_k = w_k, & \\
\dot w_k =  H\star(\mu+\mu_m)(y_k,w_k)+u_k, & k=1,\dots m, \quad t\in[0,T]\\
\partial_t\mu+v \cdot \nabla_x \mu = \nabla_v \cdot \left [ \left ( H \star (\mu +\mu_m) \right ) \mu \right ],
\end{cases}
\label{e-infdim}
\end{equation}
with control $u$, where $\mu_m$ is the time-dependent atomic measure as in \eqref{atomicleaders}, if
\begin{itemize}
\item[(i)] the measure $\mu$ is equi-compactly supported in time, i.e., there exists $R>0$ such that $\supp(\mu(t))\subset B(0,R)$ for all $t\in[0,T]$;
\item[(ii)] the solution is continuous {in time} with respect to the following metric in $\mathcal X$
\begin{equation}
\|(y,w,\mu)-(y',w',\mu')\|_{\mathcal X}:=\frac{1}{m}\sum_{k=1}^m(|y_k-y'_k|+|w_k-w'_k|)+\mathcal W_1(\mu,\mu'),
\label{e-dX}
\end{equation}
where $\mathcal W_1(\mu,\mu')$ is the 1-Wasserstein distance in $\PP(\R^{2d})$;
\item[(iv)] the $(y,w)$ coordinates define a Carath{\'e}odory solution of the following controlled problem with interaction kernel $H$, control $u(\cdot)$, and the external field $H\star\mu$:
\begin{equation} \label{s-findim}
\begin{cases}
\dot y_k = w_k, & \\
\dot w_k =  H\star(\mu+ \mu_m)(y_k,w_k)+u_k, & k=1,\dots m,, \quad t\in[0,T]
\end{cases}
\end{equation}
\item[(v)] the $\mu$ component satisfies
\begin{equation}\label{solution}
\frac{d}{dt}\int_{\R^{2d}}\phi(x,v)\,d\mu(t)(x,v)= \int_{\R^{2d}}\nabla \phi(x,v)\cdot \omega_{H, \mu, y,w}(t,x,v)\,d\mu(t)(x,v)
\end{equation}
for every $\phi \in C^\infty_c(\R^d \times \R^d)$, in the sense of distributions, where $\omega_{H,\mu,y,w}(t,x,v):[0,T]\times\R^d\times\R^d\to \R^d\times\R^d$ is the time-varying vector field defined as follows
\begin{equation}
\omega_{H, \mu, y,w}(t,x,v):=(v, H\star \mu(t)(x,v)+ H\star \mu_m(t)(x,v)).
\label{e-w}
\end{equation}
\end{itemize}
Let moreover $(y^0,w^0,\mu^0)\in \mathcal X$ be given, with {$\mu^0 \in \mathcal P_1(\mathbb R^{2d})$ of bounded support}. We say that $(y,w,\mu):[0,T]\to \mathcal X$ is a solution of \eqref{e-infdim} with initial data $(y^0,w^0,\mu^0)$ and control $u$ if it is a solution of \eqref{e-infdim} with control $u$ and it satisfies $(y(0),w(0),\mu(0))=(y^0,w^0,\mu^0)$.
\end{definition}

{Following} the well-known arguments in \cite[Section 8.1]{AGS}, once $\mu_m(t)$ is a fixed time-dependent atomic measure of the type \eqref{atomicleaders}, a measure $\mu(t)$ is a weak equi-compactly supported solution of
\begin{equation}\label{onlyPDE}
\partial_t\mu+v \cdot \nabla_x \mu = \nabla_v \cdot \left [ \left ( H \star (\mu +\mu_m) \right ) \mu \right ]
\end{equation}
in the sense of (v) in the above definition if and only if it satisfies (i) and the measure-theoretical fixed point equation
\begin{equation}\label{pushf}
\mu(t)=({\mathcal T}^{\mu,\mu_m}_t)_\sharp\mu^0,
\end{equation}
with $\mu_0:=\mu(0)$ and ${\mathcal T}^{\mu,\mu_m}_t$ is the flow function defined by \eqref{definitflow} in the Appendix. Here $({\mathcal T}^{\mu,\mu_m}_t)_\sharp$ denotes the push-forward of $\mu_0$ through ${\mathcal T}^{\mu,\mu_m}_t$.
\\

Before actually proving the existence of solutions of \eqref{e-infdim} as in Definition \ref{soldef}, it will be convenient to address the stability of the system \eqref{e-infdim} first.

\begin{proposition}
Let $u \in L^1([0,T],\U)$ be a given fixed control for \eqref{e-infdim} and two solutions $(y^1,w^1,\mu^1)$ and $(y^2,w^2,\mu^2)$ of  \eqref{e-infdim} relative to the control $u$ and 
given respective initial data $(y^{0,i},w^{0,i},\mu^{0,i}) \in \mathcal X$,  with $\mu^{0,i}$ compactly supported, $i=1,2$. Then there exists a constant $C_T>0$ such that
\begin{equation}
\| (y^1(t),w^1(t),\mu^1(t)) - (y^2(t),w^2(t),\mu^2(t)) \|_{\mathcal X} \leq C_T  \| (y^{0,1},w^{0,1},\mu^{0,1}) - (y^{0,2},w^{0,2},\mu^{0,2}) \|_{\mathcal X}, \quad \mbox{for all } t \in [0,T]
\end{equation}
\end{proposition}

\begin{proof} We show the stability estimate by  chaining the stability of \eqref{s-findim} with the one of  \eqref{onlyPDE}.
Let us first address the stability of \eqref{s-findim} given $\mu^1,\mu^2$. By integration we have 
\begin{equation}\label{firstest}
|y_k^1(t) - y_k^2(t)| \leq |y_k^{0,1} - y_k^{0,2}| + \int_0^t |w_k^{1}(s) - w_k^{2}(s)| ds,
\end{equation}
and, by Lemma \ref{secstim}, there exists a constant $L_R >0$, such that  
\begin{eqnarray}
 |w_k^1(t) - w_k^2(t)| &\leq& |w_k^{0,1} - w_k^{0,2}| + \int_0^t \left ( |(H\star{\mu^1}(y_k^1(s),w_k^1(s)) - (H\star{\mu^2}(y_k^2(s),w_k^2(s)) | \right . \nonumber \\ 
&& \phantom{XXXXXXXXXX} \left. +(H\star{\mu^1_m}(y_k^1(s),w_k^1(s)) - (H\star{\mu^2_m})(y_k^2(s),w_k^2(s)) | ds \right ), \nonumber \\
&\leq&  |w_k^{0,1} - w_k^{0,2}| + L_R \int_0^t \Big [ \left ( \frac{1}{m} \sum_{k=1}^m |y_k^1(s) - y_k^2(s)|+|w_k^1(s) - w_k^2(s)| \right ) \nonumber \\
&& \phantom{XXXXXXXXXX}  +\mathcal W_1(\mu^1(s),\mu^2(s)) \Big ] ds \label{secondest}
\end{eqnarray}
Now we consider the stability of \eqref{onlyPDE} given $\mu^1_m,\mu^2_m$. In view of the representation \eqref{pushf} of solutions by means of mass transportation, there exist  constants $\mathcal L>0$, $L_R>0$, and $\rho >0$
such that
\begin{eqnarray}
\mathcal W_1(\mu^1(t), \mu^2(t)) &=& \mathcal W_1(({\mathcal T}^{\mu^1,\mu_m^1}_t)_\sharp\mu^{0,1}, ({\mathcal T}^{\mu^2,\mu_m^2}_t)_\sharp\mu^{0,2})  \nonumber \\
&\leq & W_1(({\mathcal T}^{\mu^1,\mu_m^1}_t)_\sharp\mu^{0,1}, ({\mathcal T}^{\mu^1,\mu_m^1}_t)_\sharp\mu^{0,2}) + W_1(({\mathcal T}^{\mu^1,\mu_m^1}_t)_\sharp\mu^{0,2}, ({\mathcal T}^{\mu^2,\mu_m^2}_t)_\sharp\mu^{0,2})  \nonumber \\
&\leq & \mathcal L \mathcal W_1 (\mu^{0,1},\mu^{0,2}) + \mathcal W_1(({\mathcal T}^{\mu^1,\mu_m^1}_t)_\sharp\mu^{0,2}, ({\mathcal T}^{\mu^2,\mu_m^2}_t)_\sharp\mu^{0,2})  \nonumber \\
&\leq & \mathcal L \mathcal W_1 (\mu^{0,1},\mu^{0,2}) + \| {\mathcal T}^{\mu^1,\mu_m^1}_t - {\mathcal T}^{\mu^2,\mu_m^2}_t \|_{L^\infty(B(0,R))} \nonumber \\
&\leq & \mathcal L \mathcal W_1 (\mu^{0,1},\mu^{0,2}) + \int_0^t e^{\mathcal L (s-t)} \|(H\star \mu^1(s)-H\star \mu^2(s))  \nonumber \\
&& \phantom{XXXXXXXXXXX} + (H\star \mu^1_m(s)-H\star \mu^2_m(s))\|_{L^\infty(B(0,\rho))}\,ds  \nonumber  \\
&\leq & \mathcal L \mathcal W_1 (\mu^{0,1},\mu^{0,2}) + L_R \int_0^t e^{\mathcal L (s-t)} \Big [ \left ( \frac{1}{m} \sum_{k=1}^m |y_k^1(s) - y_k^2(s)|+|w_k^1(s) - w_k^2(s)| \right )  \nonumber  \\
&&  \phantom{XXXXXXXXXXX}  + \mathcal W_1(\mu^1(s),\mu^2(s)) \Big ]\,ds, \label{thirdest}
\end{eqnarray}
where we first applied the triangle inequality, in the second inequality we used the Lipschitz continuity of the flow map ${\mathcal T}^{\mu^1,\mu_m^1}_t$ given by \eqref{contflow} for $\mu^1 = \mu_2$ and $\mu_m^1=\mu_m^2$, and 
 Lemma \ref{primstim} also for the third inequality, the fourth inequality is again a consequence of \eqref{contflow}, and the last one again due to an application of Lemma \ref{secstim}. By combining \eqref{firstest}, \eqref{secondest}, and \eqref{thirdest}, and recalling the definition of the norm $\| \cdot\|_{\mathcal X}$, we easily recognize and conclude  the estimate
\begin{eqnarray*}
&&\| (y^1(t),w^1(t),\mu^1(t)) - (y^2(t),w^2(t),\mu^2(t)) \|_{\mathcal X} \leq\\
& \leq &  C_0 \left ( \| (y^{0,1},w^{0,1},\mu^{0,1}) - (y^{0,2},w^{0,2},\mu^{0,2}) \|_{\mathcal X} + \int_0^t \| (y^1(s),w^1(s),\mu^1(s)) - (y^2(s),w^2(s),\mu^2(s)) \|_{\mathcal X} ds \right ),
\end{eqnarray*}
for a suitable constant $C_0>0$ depending on $\mathcal L, L_R, T$. An application of Gronwall's inequality concludes the stability estimate.
\end{proof}

This latter result also implies that, once a control $u \in L^1([0,T],\U)$ is fixed, the solution of \eqref{e-infdim}, if it exists, is uniquely determined by the initial conditions.
We shall derive now the existence of solutions of \eqref{e-infdim} in the sense of Definition \ref{soldef} by a limit process for $N \to \infty$ where we allow for a variable control
$u_N$ depending on $N$.

\begin{theorem} \label{p-infdimsol} Let $(y^0,w^0,\mu^0)\in {\mathcal X}$ be given, with $\mu^0$ of bounded support in $B(0,R)$, for $R>0$. Define a sequence $(\mu_N^0)_{N \in \mathbb N}$ of atomic probability measures equi-compactly supported in $B(0,R)$ such that each $\mu^0_N$ is given by $\mu^0_N:=\sum_{i=1}^N \delta_{x^0_{i,N},v^0_{i,N}}$ and  $\lim_{N \to \infty} \mathcal W_1(\mu_N^0,\mu^0)=0$. Fix now a weakly convergent sequence  $(u_N)_{N \in \mathbb N} \subset  L^1([0,T],\U)$ of control functions, i.e.,  $u_N \rightharpoonup u_*$ in $L^1([0,T],\U)$. For each initial datum $\zeta_N^0=(y^0,w^0,x^0_{N},v^0_{N})$ depending on $N$, denote with $\zeta_N(t)=(y_N(t),w_N(t),\mu_N(t)):=(y_N(t),w_N(t),x_N(t),v_N(t))$ the unique solution of the finite-dimensional control problem \eqref{e-findim} with control $u_N$. {(Here we apply the identification of the trajectories $(x(t),v(t))$ and the measure $\mu_N(t)$ by means of \eqref{muNfoll}.)}  Then, the sequence $(y_N,w_N,\mu_N)$ converges in $C^0([0,T],{\mathcal X})$ to some $(y_*,w_*,\mu_*)$, which is a solution of \eqref{e-infdim} with initial data $(y^0,w^0,\mu^0)$ and control $u_*$, in the sense of Definition \ref{soldef}
\end{theorem}
\begin{proof} Since the initial data $\zeta_N^0$ are equi-compactly supported, the trajectories   $\zeta_N(t)=(y_N(t),w_N(t),\mu_N(t))$ are equibounded and equi-Lipschitz continuous in $C^0([0,T],{\mathcal X})$, because of \eqref{wasserest}, combined with \eqref{growth} and \eqref{lipcont}. By an application of the Ascoli-Arzel\`a theorem for functions on $[0,T]$ and values in the complete metric space ${\mathcal X}$, there exists a subsequence, again denoted by $\zeta_N(\cdot)=(y_N(\cdot),w_N(\cdot),\mu_N(\cdot))$ converging uniformly to a limit $\zeta_*=(y_*(\cdot),w_*(\cdot),\mu_*(\cdot))$, which is also equi-compactly supported in a ball $B(0,R_T)$ for a suitable $R_T>0$.  Due to equi-Lipschitz continuity in time of the trajectories and the continuity of the Wasserstein distance, we also obtain
$$
\| \zeta_*(t_2) - \zeta_*(t_1) \|_{\mathcal X} {= \lim_N} \| \zeta_N(t_2) - \zeta_N(t_1) \|_{\mathcal X} \leq L_T |t_2 - t_1|,
$$
for all $t_1,t_2 \in [0,T]$, where $L_T>0$ is a uniform Lipschitz constant. Hence, the limit trajectory $\zeta_*$ belongs as well to $C^0([0,T],{\mathcal X})$. 
\\

We need now to show that $\zeta_*$ is a solution of \eqref{e-infdim} in the sense of Definition \ref{soldef}. We first verify that $(y_*,w_*)$ is a solution of the ODEs part of  \eqref{e-infdim}
for $\mu = \mu_*$.\\
To this end we observe that the limit $\zeta_N \to \zeta_*$ in particular specifies into 
\begin{equation}\label{unif}
\left \{
\begin{array}{ll}
\xi_N \rightrightarrows \xi_*, & \mbox{ in } [0,T],\\
\dot \xi_N  \rightharpoonup \dot \xi_*, & \mbox{ in } L^1([0,T],\mathbb R^{2d}),\\
\end{array}
\right .
\end{equation}
where $\xi_N(t)=(y_N(t),w_N(t))$ and $\xi_*=(y_*,w_*)$ and 
\begin{equation}\label{convwas}
\lim_N \mathcal W_1 (\mu_N(t), \mu_*(t)) =0, 
\end{equation}
uniformly with respect to $t \in [0,T]$. In particular the limits \eqref{unif} imply $\dot y_{k,*}(t) = w_{k,*}(t)$ in $[0,T]$ for all $k=1,\dots,m$. We shall now show that  $(y_*(t),w_*(t))$ is actually the Carath{\'e}odory solution of \eqref{s-findim}
by verifying also its second equation.

Let us denote now
$$
\mu_{m,N}(t) = \frac{1}{m} \sum_{k=1}^m \delta_{(y_{k,N}(t),w_{k,N}(t))} \quad \mbox{and }  \mu_{m,*}(t) = \frac{1}{m} \sum_{k=1}^m \delta_{(y_{k,*}(t),y_{k,*}(t))}.
$$
As a consequence of \eqref{wasserest}, Lemma  \ref{secstim} in the Appendix, and the uniform convergence of the trajectories we have that 
\begin{equation}\label{Wasser}
\mathcal W_1 (\mu_{m,N}(t), \mu_{m,*}(t))\to 0
\end{equation}
as $N\to +\infty$, uniformly in $t\in [0,T]$. By \eqref{convwas}, \eqref{Wasser}, and the linear growth of $H$ we deduce
\begin{equation}\label{potr}
H \star (\mu_N + \mu_{m,N}) (y_{k,N},v_{k,N}) \rightrightarrows  H \star (\mu_* + \mu_{m,*}) (y_{k,*},w_{k,*}), \mbox{ in } [0,T], \quad \mbox{ for } N \to \infty,
\end{equation}
again by applying Lemma \ref{secstim} in the Appendix. 

To prove that $(y_*(t),w_*(t))$ is actually the Carath{\'e}odory solution of \eqref{s-findim}, we have only to show that for all $k=1,\dots,m$ one has
\begin{equation}\label{claim}
\dot w_{k,*} = ( H \star \mu_*  + \mu_{m,*})(y_{k,*},w_{k,*}) + u_{k,*}.
\end{equation}
This is clearly equivalent to the following: for every $\eta \in \R^d$ and every $\hat t\in [0,T]$ it holds
\begin{equation}\label{claim2}
\eta \cdot \int_0^{\hat t}\dot w_{k,*}(t)\,dt = \eta \cdot \int_0^{\hat t}[( H \star \mu_*(t)  + \mu_{m,*}(t))(y_{k,*}(t),w_{k,*}(t)) + u_{k,*}(t)]\,dt,
\end{equation} 
which follows from the weak $L^1$ convergence of $\dot w_{k,N}$ to $\dot w_{k,*}$ and of $u_N$ to $u_*$ for $N\to +\infty$, and from \eqref{potr}. 
\\

We are now left with verifying that $\mu_*$ is a solution of \eqref{onlyPDE} for $\mu_m = \mu_{m,*}$ in the sense of Definition \ref{soldef} (v).  For all $\hat t \in [0,T]$ and for all $\phi \in C_c^1(\mathbb R^{2d})$ we infer that
$$
\langle \phi, \mu_N(\hat t) - \mu_N(0) \rangle = \int_0^{\hat t} \left [ \int_{\mathbb R^{2d}} \nabla \phi(x,v) \cdot w_{H,\mu_N,y_N,w_N}(t,x,v) d \mu_N(t,x,v) \right ] dt,
$$
which is verified by considering the differentiation
\begin{eqnarray*} \frac{d}{dt} \langle \phi, \mu_N(t)\rangle &=&  \frac{1}{N} \frac{d}{dt} \sum_{i=1}^N \phi(x_i(t),v_i(t))\\
&=& \frac{1}{N} \left [ \sum_{i=1}^N \nabla_x \phi(x_i(t),v_i(t)) \cdot \dot x_i(t) + \sum_{i=1}^N \nabla_v \phi(x_i(t),v_i(t)) \cdot \dot v_i(t) \right],
\end{eqnarray*}
and directly applying the substitutions as in \eqref{e-findim} for the followers variables $(x,v)$.  Moreover,
\begin{equation}\label{prima}
\lim_{ N \to \infty} \langle \phi, \mu_{N}(\hat t) -\mu_{N}(0) \rangle =  \langle \phi, \mu_*(\hat t) -\mu^0 \rangle,
\end{equation}
for all $\phi \in C_c^1(\mathbb R^{2d})$.  By possibly extracting an additional subsequence, by weak-$*$ convergence, and the dominated convergence theorem, we obtain the limit
\begin{equation}\label{terza}
\lim_{k \to \infty} \int_0^{\hat t} \int_{\mathbb R^{2d}} (\nabla_v \phi(x,v) \cdot v ) d \mu_{N_k}(t)(x,v) dt =  \int_0^{\hat t} \int_{\mathbb R^{2d}} (\nabla_v \phi(x,v) \cdot v) d \mu_*(t)(x,v) dt,
\end{equation}
for all $\phi \in C_c^1(\mathbb R^{2d})$. By Lemma \ref{secstim} in Appendix we also have that for every $\rho >0$ 
$$
\lim_{k \to \infty}  \|H\star (\mu_{N_k}(t)+ \mu_{m,N}(t)) -H\star (\mu_*(t)+ \mu_{m,*}(t))\|_{L^\infty(B(0,\rho))} =0,
$$
and, as $\phi \in C_c^1(\mathbb R^{2d})$ has compact support, it follows that
$$
\lim_{k \to \infty}  \| \nabla_v \phi \cdot (H\star (\mu_{N_k}(t)+ \mu_{m,N_k}(t)) -H\star (\mu_*(t)+ \mu_{m,*}(t))) \|_\infty =0.
$$
Denote with $\mathcal L^1\llcorner_{[0,\hat t]}$ the Lebesgue measure on the time interval $[0,\hat t]$. Since the product measures $\mathcal L^1\llcorner_{[0,\hat t]} \times \frac{1}{\hat t} \mu_{N_k}(t)$ converge in $\mathcal P_1([0,\hat t] \times \mathbb R^{2d})$ to $\mathcal L^1\llcorner_{[0,\hat t]} \times \frac{1}{\hat t} \mu_*(t)$, we finally get
\begin{eqnarray}
&&\lim_{k \to \infty} \int_0^{\hat t} \int_{\mathbb R^{2d}} (\nabla_v \phi(x,v) \cdot H\star (\mu_{N_k}(t)+ \mu_{m,N_k}(t))) d \mu_{N_k}(t)(x,v) dt \nonumber  \\
&=&  \int_0^{\hat t} \int_{\mathbb R^{2d}} (\nabla_v \phi(x,v) \cdot H\star (\mu_*(t)+ \mu_{m,*}(t))) d \mu_*(t)(x,v) dt. \label{quarta}
\end{eqnarray}
The statement now follows by combining \, \eqref{prima},  \eqref{terza}, and \eqref{quarta}.

 \end{proof}

\begin{remark}\label{fullsequence}
In the proof of the previous theorem we consider a converging subsequence of $\zeta_N$ after application of the Ascoli-Arzel\`a theorem. Let us stress that in view of the uniqueness of the solution of \eqref{e-infdim}, we do not
need to restrict us to a subsequence, but we can infer the convergence of the entire sequence $\zeta_N$ to the solution of \eqref{e-infdim}. This observation will play an important role below, when we shall prove the $\Gamma$-convergence of
finite dimensional optimal control problems constrained by the ODE system \eqref{e-findim} to the infinite dimensional optimal control problem constrained by the ODE-PDE system \eqref{e-infdim}.
\end{remark}
\section{The finite-dimensional optimal control problem}

We state the following assumptions:
\begin{itemize}
\item[(L)]  Let $L: \mathcal X \to \mathbb R_+$
be a continuous function with respect to the distance induced on ${\mathcal X}$ by the norm $\| \cdot \|_{\mathcal X}$;
\end{itemize}
Given $N\in \mathbb N$ and an initial {datum} $(y_1(0),\ldots,y_m(0),w_1(0),\ldots,w_m(0),x_1(0), \dots, x_N(0), v_1(0), \dots, v_N(0)) \in (\mathbb R^d)^m \times(\mathbb R^d)^m \times(\mathbb R^d)^N \times (\mathbb R^d)^N$, 
we consider the following optimal control problem:
\begin{equation}\label{mainmod3}
\min_{u=(u_1,\dots,u_k) \in L^1([0,T],\U)} \int_{0}^T  \left\{ L(y(t),w(t),\mu_N(t))   +  \frac{1}{m} \sum_{k=1}^m |u_k(t)|  \right \}dt,
\end{equation}
where 
\begin{equation}\label{muenne}
\mu_m(t)(x,v) = \frac{1}{m} \sum_{k=1}^m \delta_{(y_k(t),w_k(t))}(x,v)\mbox{~~~~and~~~~}\mu_N(t)(x,v) = \frac{1}{N} \sum_{i=1}^N \delta_{(x_i(t),v_i(t))}(x,v)
\end{equation}
are the time dependent  atomic measures supported on the phase space trajectories $(y_k(t),w_k(t)) \in \mathbb R^{2 d}$, for $k=1,\dots m$ and $(x_i(t),v_i(t)) \in \mathbb R^{2 d}$, for $i=1,\dots N$, respectively, constrained by being the solution of the system
\begin{equation}
\begin{cases}
\dot y_k = w_k, & \\
\dot w_k =  H\star \mu_N(y_k,w_k) + H \star\mu_m(y_k,w_k) + u_k\, & k=1,\dots m, \quad t \in [0,T],\\
\dot x_i = v_i, & \\
\dot v_i =  H\star \mu_N(x_i,v_i) + H \star\mu_m(x_i,v_i)\, & i=1,\dots N, \quad t \in [0,T],
\end{cases}
\label{e-findim2}
\end{equation}
for a given datum $\zeta^0=(y^0,w^0,x^0,v^0) \in \mathbb R^{2d(m+N)}$ and control $u\in L^1([0,T],\mathcal{U})$.\\

Let us recall that the existence of  Carath{\'e}odory solutions of \eqref{e-findim2} for any given $u=(u_1,\dots,u_k) \in L^1([0,T],\U)$ is  ensured by Proposition \ref{exsol}.

\begin{theorem}\label{thm:4}
The finite horizon optimal control problem \eqref{mainmod3}-\eqref{e-findim2} with initial datum $\zeta^0=(y^0,w^0,x^0,v^0) \in \mathbb R^{4d}$ has solutions.
\end{theorem}
\begin{proof} For the sake of self-containedness and broad readability, we just sketch  briefly the proof of this statement, which follows from very classical results in optimal control, see, e.g., {\cite[Theorem 5.2.1]{BressanPiccoli}}.
Let $(u^n)_{n \in \mathbb N}$ be a  minimizing sequence {realizing at its limit the minimum of the cost functional in \eqref{mainmod3}}. As this sequence is necessarily bounded in $L^1([0,T],\U)$, it admits a subsequence, which we simply rename as $(u^n)_{n \in \mathbb N}$, weakly
converging to a $u^* \in L^1([0,T],\U)$. At the same time the corresponding solutions $\zeta^n(t)=(y^n(t),w^n(t),x^n(t),v^n(t))$ of \eqref{e-findim2} given the  control $u^n$ in $L^1([0,T],\U)$ are equi-bounded and equi-Lipschitz continuous in time, thanks to an argument identical to  the one given at the beginning of the proof of Proposition \ref{p-infdimsol}. We similarly conclude that $\zeta^n$ has a subsequence, again not relabeled, converging uniformly to a trajectory $\zeta^*$ which is actually the solution of \eqref{e-findim2} given the control $u^*$ in $L^1([0,T],\U)$. The uniform convergence of the trajectories and their compact support also allow us to conclude by the use of condition (L) that 
\begin{eqnarray*}
&&\lim_{n \to +\infty} \int_{0}^T   L(y^n(t),w^n(t),\mu_N^n(t)) dt = \int_{0}^T   L(y^*(t),w^*(t),\mu_N^*(t)) dt,
\end{eqnarray*}
and the weak convergence of $(u^n)_{n \in \mathbb N}$ to $u^* \in L^1([0,T],\U)$ implies the lower-semicontinuity of the norm
$$
\lim \inf_{n \to +\infty}  \frac{1}{m} \int_{0}^T \sum_{k=1}^m |u_k^n(t)|  dt \geq \frac{1}{m} \int_{0}^T \sum_{k=1}^m |u_k^*(t)|  dt.
$$
We conclude by these two limits that $u^*$ is an optimal control for  \eqref{mainmod3}-\eqref{e-findim2}.

\end{proof}

\section{The $\Gamma$-limit to the infinite-dimensional optimal control problem}

We shall now recall the concept of $\Gamma$-limit, which, together with the mean-field limit established by Theorem \ref{p-infdimsol} will allow us  to prove that solutions of the optimal control problems  \eqref{mainmod3}-\eqref{e-findim2} converges to  optimal controls for the system \eqref{e-infdim}.

\begin{definition} [\( \Gamma \)-convergence]
  \label{def:gamma-conv}
  \cite[Definition 4.1, Proposition 8.1]{93-Dal_Maso-intro-g-conv}
  Let \( X \) be a metrizable separable space and \( F_N \colon X \rightarrow (-\infty,\infty] \), \( N \in \mathbb{N} \) be a sequence of functionals. Then we say that \( F_N \) \emph{\( \Gamma \)-converges} to \( F \), written as \( F_N \xrightarrow{\Gamma} F \), for an \( F \colon X \rightarrow (-\infty,\infty] \), if
  \begin{enumerate}
  \item \emph{\( \liminf \)-condition:} For every \( u \in X \) and every sequence \( u_N \rightarrow u \),
    \begin{equation*}
      F(u) \leq \liminf_{N\rightarrow\infty} F_N(u_N);
    \end{equation*}
  \item \emph{\( \limsup \)-condition:} For every \( u \in X \), there exists a sequence \( u_N \rightarrow u \), called \emph{recovery sequence}, such that
    \begin{equation*}
      F(u) \geq \limsup_{N\rightarrow\infty} F_N(u_N).
    \end{equation*}
  \end{enumerate}

  Furthermore, we call the sequence \( (F_N)_N \) \emph{equi-coercive} if for every \( c \in \mathbb{R} \) there is a compact set \( K \subseteq X \) such that \( \left\{ u : F_N(u) \leq c \right\} \subseteq K \) for all \( N \in \mathbb{N} \). As a direct consequence,  {assuming \( u_N^* \in \arg \min F_N \neq \emptyset \)} for all $N \in \mathbb N$, there is a subsequence \( (u_{N_k}^*)_k \) and \( u^\ast \in X \) such that
  \begin{equation*}
    u_{N_k}^* \rightarrow u^\ast \in \arg \min F.
  \end{equation*}
\end{definition}

In the following we assume that $H$ is a function satisfying (H) so that \eqref{e-findim} and \eqref{e-infdim} are well-posed, for a given control $u$ and suitable initial conditions. In view of the definition of $\Gamma$-convergence, let us fix as our domain $X = L^1([0,T],\U)$ which, endowed with the weak $L^1$-topology, is actually a metrizable space.

Fix now an initial datum $(y^0,w^0,\mu^0) \in \mathcal X$, with $\mu^0$ compactly supported, $\supp(\mu^0) \subset B(0,R)$, $R>0$, and choose a sequence of equi-compactly supported atomic measures $\mu_N^0$, $\supp(\mu^0_N) \subset B(0,R)$, $\mu_N^0 = \frac{1}{N} \sum_{i=1}^N \delta_{(x_i^0,v_i^0)}$ such that $\mathcal W_1(\mu^0_N, \mu^0) \to 0$ for $N \to +\infty$.

We define the following functional on $X$
\begin{equation}
F(u) = \int_{0}^T  \left \{ L(y(t),w(t),\mu(t))   +  \frac{1}{m} \sum_{k=1}^m |u_k(t)|  \right \}dt,
\label{Ffun}
\end{equation}
where the triplet $(y,w,\mu)$ defines the unique solution of \eqref{e-infdim}  with initial datum $(y^0,w^0,\mu^0)$ and control $u$, i.e.,
\begin{equation}
\begin{cases}
\dot y_k = w_k, & \\
\dot w_k =  H\star(\mu+\mu_m)(y_k,w_k)+u_k, & k=1,\dots m, \quad t\in[0,T]\\
\partial_t\mu+v \cdot \nabla_x \mu = \nabla_v \cdot \left [ \left ( H \star (\mu +\mu_m) \right ) \mu \right ],
\end{cases}
\label{e-infdim2}
\end{equation}
{in the sense of Definition \ref{soldef}.} Similarly, we define the functionals on $X$ given by
\begin{equation}
F_N(u) = \int_{0}^T  \left \{  L(y_{N}(t),w_{N}(t),\mu_N(t))  +  \frac{1}{m} \sum_{k=1}^m |u_k(t)|  \right \}dt,
\label{FNfun}
\end{equation}
where $\mu_N(t) = \frac{1}{N} \sum_{i=1}^N  \delta_{(x_{i,N}(t),v_{i,N}(t))}$ is the time-dependent atomic measure supported on the trajectories defining the
Carath{\'e}odory solution of the system
\begin{equation}
\begin{cases}
\dot y_k = w_k, & \\
\dot w_k =  H\star \mu_N(y_k,w_k) + H \star\mu_m(y_k,w_k) + u_k\, & k=1,\dots m, \quad t \in [0,T],\\
\dot x_i = v_i, & \\
\dot v_i =  H\star \mu_N(x_i,v_i) + H \star\mu_m(x_i,v_i)\, & i=1,\dots N, \quad t \in [0,T],
\end{cases}
\label{e-findim3}
\end{equation}
with initial datum $(y^0,w^0,x^0_N,v^0_N)$ and control $u$.  

\begin{remark}Observe that the choice of the functionals $F_N$ depends on the choice of the sequence $\mu^0_N$ approximating $\mu^0$.
\end{remark}
The rest of this section is devoted to the proof of the $\Gamma$-convergence
of the sequence of functionals $(F_N)_{N \in \mathbb N}$ on $X$ to the target functional $F$. Let us mention that $\Gamma$-convergence in optimal control problems has been already considered, see for instance \cite{BuDm82}, but, to our knowledge, it has been only recently specified in connection to mean-field limits in \cite{MFOC}.
\begin{theorem}\label{thm:gamma}
Let $H$ and $L$ be maps satisfying conditions (H) and (L) respectively. Given an initial datum $(y^0,w^0,\mu^0) \in \mathcal X$ and an approximating sequence $\mu^0_N$,  with $\mu^0,\mu^0_N$ {equi-compactly supported, i.e., $\supp(\mu^0) \cup \supp(\mu^0_N) \subset B(0,R)$, $R>0$, for all $N \in \mathbb N$}, then the sequence of functionals $(F_N)_{N \in \mathbb N}$ on $X=L^1([0,T],\U)$ defined in \eqref{FNfun} $\Gamma$-converges to the functional $F$ defined in \eqref{Ffun}.
\end{theorem}
\begin{proof}
Let us start by showing the $\Gamma-\liminf$ condition. Let us fix a weakly convergent sequence of controls $u_N \rightharpoonup u_*$ in $L^1([0,T],\U)$. As done in the proof of Theorem \ref{p-infdimsol} we can associate to each of these controls a sequence
of solutions $\zeta_N(t)=  (y_N(t),w_N(t),\mu_N(t)):= (y_N(t),w_N(t),x_N(t),v_N(t))$ of \eqref{e-findim3} uniformly convergent to a solution $\zeta_*(t) = (y_*(t),w_*(t),\mu_*(t))$ of \eqref{e-infdim2} in the sense of Definition \ref{soldef} with control $u_*$ and initial datum  $(y^0,w^0,\mu^0)$. In view of the fact that solutions $\zeta_N(t)$ and $\zeta_*(t)$ will have supports uniformly bounded with respect to $N$ and $t \in [0,T]$ and by the uniform convergence of trajectories $(y_N(t),w_N(t)) \rightrightarrows (y_*(t),w_*(t))$ as well as the uniform convergence $\mathcal W_1(\mu_N(t),\mu_*(t)) \to 0$ for $t \in [0,T]$, it follows from condition (L) that
\begin{eqnarray}
&&\lim_{ N \to +\infty} \int_{0}^T   L(y_{N}(t),w_{N}(t),\mu_{N}(t)) =\int_{0}^T  L (y_{*}(t),w_{*}(t),\mu_*(t)) dt. \label{lowsem1}
\end{eqnarray}
Notice that, thanks to Remark \ref{fullsequence}, here we are allowed to consider the convergence of the entire sequence $(\zeta_N)_{N \in \mathbb N}$ and we do not need to 
restrict to a subsequence (and this is a crucial issue in order to properly derive the $\Gamma-\liminf$ condition!).
By the assumed weak convergence of $(u_N)_{N \in \mathbb N}$ to $u^* \in L^1([0,T],\U)$ we obtain the lower-semicontinuity of the norm
\begin{equation}
\lim \inf_{n \to +\infty}  \frac{1}{m} \int_{0}^T \sum_{k=1}^m |u_{k,N}(t)|  dt \geq \frac{1}{m} \int_{0}^T \sum_{k=1}^m |u_{k,*}(t)|  dt. \label{lowsem2}
\end{equation}
By combining \eqref{lowsem1} and \eqref{lowsem2}, we immediately obtain the $\Gamma-\liminf$ condition
$$
\liminf_{N \to \infty} F_N(u_N) \geq F(u_*).
$$
We need now to address the $\Gamma-\limsup$ condition. Let us fix $u_*$ and we consider the trivial recovery sequence $u_N \equiv u_*$ for all $N \in \mathbb N$. Similarly as above for the argument of the $\Gamma-\liminf$ condition, we can associate to each of these controls a sequence
of solutions $\zeta_N(t)=  (y_N(t),w_N(t),\mu_N(t)):= (y_N(t),w_N(t),x_N(t),v_N(t))$ of \eqref{e-findim3} uniformly convergent to a solution $\zeta_*(t) = (y_*(t),w_*(t),\mu_*(t))$ of \eqref{e-infdim2} in the sense of Definition \ref{soldef} with control $u$ and initial datum  $(y^0,w^0,\mu^0)$ and we can similarly conclude the limit \eqref{lowsem1}. Additionally, being the sequence $(u_N)_{N \in \mathbb N}$ trivially a constant sequence   we have
\begin{equation}
\lim \inf_{n \to +\infty}  \frac{1}{m} \int_{0}^T \sum_{k=1}^m |u_{k,N}(t)|  dt = \frac{1}{m} \int_{0}^T \sum_{k=1}^m |u_{k,*}(t)|  dt. \label{lowsem3}
\end{equation}
Hence, combining \eqref{lowsem1} and \eqref{lowsem3} we can easily infer
$$
\limsup_{N \to \infty} F_N(u_N) = \lim_{N \to \infty} F_N(u_*) =  F(u_*).
$$
\end{proof}

\begin{corollary}
Let $H$ and $L$ be maps satisfying conditions (H) and (L) respectively. Given an initial datum $(y^0,w^0,\mu^0) \in \mathcal X$,  with $\mu^0$ compactly supported, $\supp(\mu^0) \subset B(0,R)$, $R>0$, the optimal control problem
\begin{equation}
\label{optcontoo}
\min_{u \in L^1([0,T],\U)} \int_{0}^T  \left \{ L(y(t),w(t),\mu(t))  +  \frac{1}{m} \sum_{k=1}^m |u_k(t)|  \right \}dt,
\end{equation}
has solutions, where the triplet $(y,w,\mu)$ defines the unique solution of \eqref{e-infdim} with initial datum $(y^0,w^0,\mu^0)$ and control $u$ of
\begin{equation}
\begin{cases}
\dot y_k = w_k, & \\
\dot w_k =  H\star(\mu+\mu_m)(y_k,w_k)+u_k, & k=1,\dots m, \quad t\in[0,T]\\
\partial_t\mu+v \cdot \nabla_x \mu = \nabla_v \cdot \left [ \left ( H \star (\mu +\mu_m) \right ) \mu \right ],
\end{cases}
\end{equation}
{in the sense of Definition \ref{soldef},} and
\begin{equation}
\mu_m(t) = \frac{1}{m} \sum_{k=1}^n \delta_{(y_k(t),w_k(t))}.
\end{equation} 
Moreover, solutions to \eqref{optcontoo} can be constructed as weak limits $u^*$ of sequences of optimal controls $u_{N}^*$ of the finite dimensional problems
\begin{equation}
\min_{u \in L^1([0,T],\U)} \int_{0}^T  \left \{ L(y_{N}(t),w_{N}(t),\mu_N(t))  +  \frac{1}{m} \sum_{k=1}^m |u_k(t)|  \right \}dt,
\label{FNfun2}
\end{equation}
where $\mu_N(t) = \frac{1}{N} \sum_{i=1}^N  \delta_{(x_{i,N}(t),v_{i,N}(t))}$ and $\mu_{m,N}(t) = \frac{1}{m} \sum_{k=1}^m  \delta_{(y_{k,N}(t),w_{k,N}(t))}$ are the time-dependent atomic measures supported on the trajectories defining the
solution of the system
\begin{equation}
\begin{cases}
\dot y_k = w_k, & \\
\dot w_k =  H\star \mu_N(y_k,w_k) + H \star\mu_{m,M}(y_k,w_k) + u_k\, & k=1,\dots m, \quad t \in [0,T],\\
\dot x_i = v_i, & \\
\dot v_i =  H\star \mu_N(x_i,v_i) + H \star\mu_{m,M}(x_i,v_i)\, & i=1,\dots N, \quad t \in [0,T],
\end{cases}
\label{e-findim4}
\end{equation}
with initial datum $(y^0,w^0,x^0_N,v^0_N)$, control $u$, and   $\mu_N^0 = \frac{1}{N} \sum_{i=1}^N \delta_{(x_i^0,v_i^0)}$ is such that
$\mathcal W_1(\mu^0_N, \mu^0) \to 0$ for $N \to +\infty$.
\end{corollary}
\begin{proof}
Notice that the optimal controls $u_N^*$ of the finite dimensional optimal control problems \eqref{FNfun2}-\eqref{e-findim4} belongs to $X=  L^1([0,T],\U)$, which is a compact set with respect to the weak topology of $L^1$.
Hence $(u_N^*)_{N \in \mathbb N}$ admits a  subsequence, which we do not relabel, weakly convergent to some $u^* \in L^1([0,T],\U)$.  Moreover, as done in the proof of Theorem \ref{p-infdimsol} we can associate to each of these controls a sequence
of solutions $\zeta_N(t)=  (y_N(t),w_N(t),\mu_N(t)):= (y_N(t),w_N(t),x_N(t),v_N(t))$ of \eqref{e-findim3} uniformly convergent to a solution $\zeta_*(t) = (y_*(t),w_*(t),\mu_*(t))$ of \eqref{e-infdim2} in the sense of Definition \ref{soldef} with control $u^*$. In order to conclude that $u^*$ is an optimal control for \eqref{e-findim3} we need to show that it is actually a minimizer of $F$. For that we use the fact that $F$ is the $\Gamma$-limit of the sequence $(F_N)_{N \in \mathbb N}$ as proved in Theorem \ref{thm:gamma}.
Let $u \in X$ be an arbitrary control and let $(u_N)_{N \in \mathbb N}$ be a recovery sequence given by the $\Gamma-\limsup$ condition, so that
\begin{eqnarray}
F(u) \geq \limsup_{N \to \infty} F_N(u_N).\label{firstin}
\end{eqnarray}
By using now the optimality of $(u_N^*)_{N \in \mathbb N}$
\begin{eqnarray}
 \limsup_{N \to \infty} F_N(u_N) \geq \limsup_{N \to \infty} F_N(u_N^*) \geq \liminf_{N \to \infty} F_N(u_N^*).\label{secondin}
\end{eqnarray}
Applying the  $\Gamma-\liminf$ condition yields
\begin{eqnarray}
  \liminf_{N \to \infty} F_N(u_N^*) \geq F(u^*).\label{thirdin}
\end{eqnarray}
By chaining the inequalities \eqref{firstin}-\eqref{secondin}-\eqref{thirdin} we eventually obtain that
$$
F(u) \geq F(u^*), \quad \mbox{ for all } u \in X,
$$
or that $u^*$ is an optimal control.
\end{proof}

\begin{remark} {Observe that the previous result does not state uniqueness of the optimal control for the infinite dimensional problem. Moreover, in general, we cannot ensure that such optimal controls are always limits of sequences of optimal controls of  \eqref{FNfun2}-\eqref{e-findim4}.}
\end{remark}

\section{Appendix}

For the reader's convenience we start by briefly recalling some well-known results about solutions to Carath{\'e}odory differential equations. We fix an interval $[0,T]$ on the real line, and let $n\ge 1$.
Given a domain $\Omega \subset \R^n$, a Carath{\'e}odory function $g\colon[0,T]\times \Omega \to \R^n$, and $0<\tau \le T$, a function $y\colon [0,\tau]\to \Omega$ is called a solution of the Carath{\'e}odory differential equation
\begin{equation}\label{cara}
\dot y(t)=g(t, y(t))
\end{equation}
on $[0,\tau]$ if and only if $y$ is absolutely continuous and \eqref{cara} is satisfied a.e.\ in $[0,\tau]$.
The following existence and uniqueness result holds.
\begin{theorem}\label{cara2}
Consider an interval $[0,T]$ on the real line, a domain $\Omega \subset \R^n$, $n\ge 1$, and a Carath{\'e}odory function $g\colon[0,T]\times \Omega \to \R^n$. Assume that there exists a constant $C>0$ such that
$$
|g(t,y)|\le C
$$
for a.e.\ $t \in [0,T]$ and every $y \in \Omega$. Then, given $y_0 \in \Omega$, there exists $0<\tau \le T$ and a solution $y(t)$ of \eqref{cara} on $[0,\tau]$ satisfying $y(0)=y_0$. 

If in addition there exists another constant $L>0$ such that
\begin{equation}\label{cara3}
|g(t,y_1)-g(t, y_2)|\le L|y_1-y_2|
\end{equation}
for a.e.\ $t \in [0,T]$ and every $y_1$, $y_2 \in \Omega$, the solution is uniquely determined on $[0,\tau]$ by the initial condition $y_0$.
\end{theorem}

\begin{proof}
See, for instance, \cite[Chapter 1, Theorems 1 and 2]{Fil}.
\end{proof}

Also the global existence theorem and a Gronwall estimate on the solutions can be easily generalized to this setting.

\begin{theorem}\label{cara-global}
Consider an interval $[0,T]$ on the real line and a Carath{\'e}odory function $g\colon[0,T]\times \R^n \to \R^n$. Assume that there exists a constant $C>0$ such that
\begin{equation}\label{ttz}
|g(t,y)|\le C(1+|y|)
\end{equation}
for a.e.\ $t \in [0,T]$ and every $y \in \R^n$. Then, given $y_0 \in \R^n$, there exists a solution $y(t)$ of \eqref{cara} defined on the whole interval $[0,T]$ which satisfies $y(0)=y_0$. Any solution satisfies
\begin{equation}\label{gron}
|y(t)|\le \Big(|y_0|+ C t\Big) \,e^{C t}
\end{equation}
for every $t \in [0,T]$.

If in addition, for every relatively compact open subset of $\R^n$, \eqref{cara3} holds, the solution is uniquely determined on $[0,T]$ by the initial condition $y_0$.
\end{theorem}

\begin{proof}
Let $C_0:= (|y_0|+ C T) \,e^{C T}$. Take a ball $\Omega \subset \R^p$  centered at $0$ with radius strictly greater than $C_0$. Existence of a local solution defined on an interval $[0,\tau]$ and taking values in $\Omega$ follows now easily from \eqref{ttz} and Theorem \ref{cara2}. Using \eqref{ttz}, any solution of \eqref{cara} with initial datum $y_0$ satisfies
$$
|y(t)|\le |y_0|+ C t+\int_0^t C |y(s)|\,ds
$$
for every $t \in [0,\tau]$, therefore \eqref{gron} follows from Gronwall's Lemma. In particular the graph of a solution $y(t)$ cannot reach the boundary of $[0,T]\times \Omega$ unless $\tau=T$, therefore existence of a global solution follows for instance from \cite[Chapter 1, Theorem 4]{Fil}. If \eqref{cara3} holds, uniqueness of the global solution follows from Theorem \ref{cara2}.
\end{proof}

The usual results on continuous dependence on the data hold also in this setting: in particular, we will use this Lemma, following from \eqref{gron} and the Gronwall inequality.

\begin{lemma}\label{lem:gronvalla}
Let $g_1$ and $g_2\colon[0,T]\times \R^n \to \R^n$ be Carath{\'e}odory functions both satisfying \eqref{ttz} for a  constant $C>0$. Let $r>0$ and define 
$$
\rho_{r, m, T}:=\Big(r+  C T\Big) \,e^{C T}\,.
$$ 
Assume in addition that there exists a constant $L=L(\rho_{r, m, T})>0$
$$
|g_1(t, y_1)-g_1(t, y_2)|\le L|y_1-y_2|
$$
for every $t \in [0, T]$ and every $y_1$, $y_2$ such that $|y_i|\le \rho_{r, m, T}$, $i=1,2$. (Notice that here we refer exclusively to $g_1$ and that the
constant $L$ actually depends on $\rho_{r, m, T}$.)
Set
$$
q(t):=\|g_1(t, \cdot)-g_2(t, \cdot)\|_{L^\infty(B(0, \rho_{r, m, T}))}\,.
$$
Then, if $\dot y_1(t)=g(t, y_1(t))$, $\dot y_2(t)=g_2(t, y_2(t))$, $|y_1(0)|\le r$ and $|y_2(0)|\le r$, one has
\begin{equation}\label{gronvalla}
|y_1(t)-y_2(t)|\le e^{L t}|y_1(0)-y_2(0)|+\int_0^t e^{\int_s^t L t}q(s)\,ds
\end{equation}
for every $t \in [0, T]$.
\end{lemma}

We mention again that $\PP(\R^n)$ denotes the space of probability measures on $\R^n$ with finite first moment. This is a metric space when endowed with the Wasserstein distance $\WW$. We recall below several useful results from \cite{CanCarRos10,MFOC} concerning Lipschitz continuity estimates for transport flows induced by the dynamics \eqref{e-findim}, which may be found in slightly different form and generality in several other papers, e.g., \cite{ambrosio,gw2,pedestrian}. The following lemma
is recalled from \cite[Lemma 6.4]{MFOC}.

\begin{lemma}\label{stimesceme}
Let $H\colon \R^n \to \R^{p}$, $n\ge p \ge 1$ be a locally Lipschitz function such that
\begin{equation}\label{hypo}
|H(y)|\le C(1+|y|), \quad \mbox{ for all } y \in \mathbb R^n, 
\end{equation}
and $\mu\colon[0,T]\to \PP(\R^n)$ be a continuous map with respect to $\WW$. Then there exists a constant $C'$ such that
\begin{equation}\label{trzu}
|H\star \mu(t) (y)|\le C'(1+|y|),
\end{equation}
for every $t \in [0, T]$ and every $y \in \R^n$. Furthermore, if
\begin{equation}\label{hypo2}
{\rm supp }\,\mu(t)\subset B(0,R),
\end{equation}
for every $t \in [0, T]$, then for every compact subset $K$ of $\R^n$ there exists a constant $L_{R,K}$ such that
\begin{equation}\label{lipt}
|H\star \mu(t) (y_1)-H\star \mu(t) (y_2)|\le L_{R,K}|y_1-y_2|,
\end{equation}
for every $t \in [0, T]$ and every $y_1$, $y_2 \in K$.
\end{lemma}

Let us consider $\mu:[0,T] \to \PP(\R^{2d})$ a continuous map with respect to $\WW$ such that ${\rm supp }\,\mu(t)\subset B(0,R)$ for all $t \in [0,T]$, and $\mu_m(t)(y,w) = \frac{1}{m} \sum_{k=1}^m \delta_{(y_k(t),w_k(t))}$ a time dependent atomic measure supported on the absolutely continuous trajectories $t \mapsto (y_k(t),w_k(t))$, $k =1,\dots,m$.
We now consider the system of ODE's on $\R^{2d}$
\begin{equation}\label{system}
\begin{cases}
\dot X(t)=V(t)\\
\dot V(t)= H\star \mu(t) (X(t), V(t))+ H\star \mu_m(t) (X(t), V(t))
\end{cases}
\end{equation}
on an interval $[0, T]$. Here $X, V$ are both mappings from $[0, T]$ to $\R^d$ and $H\colon \R^{2d} \to \R^{d}$ is a locally Lipschitz function satisfying $|H(\xi)| \leq C(1 + |\xi|)$ for all $\xi \in \mathbb R^{2d}$ and for a constant $C>0$. It follows then from these assumptions and  Lemma \ref{stimesceme} that all the hypothesis of Theorem \ref{cara-global} are satisfied. Therefore, however given $P_0:=(X_0, V_0)$ in $\R^{2d}$ there exists a unique solution $P(t):=(X(t), V(t))$ to \eqref{system} with initial datum $P_0$ defined on the whole interval $[0, T]$. We can therefore consider the family of flow maps ${\mathcal T}^{\mu,\mu_m}_t \colon \R^{2d} \to \R^{2d}$ indexed by $t\in [0,T]$ and defined by
\begin{equation}\label{definitflow}
{\mathcal T}^{\mu,\mu_m}_t(P_0):=P(t)
\end{equation}
where $P(t)$ is the value of the unique solution to \eqref{system} starting from $P_0$ at time $t=0$. The notation aims also at stressing the dependence of these flow maps on the given mappings $\mu(t),\mu_m(t)$. 
We can easily recover, as consequence of \eqref{gronvalla}, similar estimates as in \cite[Lemmas 3.7 and 3.8]{CanCarRos10}: we report the statement and a sketch of the proof of this result to allow the reader to keep track of the dependence of these constants on the data of the problem.

\begin{lemma}
Let $H\colon \R^{2d} \to \R^{d}$ be a locally Lipschitz function satisfying 
$$
|H(\xi)| \leq C(1 + |\xi|), \quad \mbox{for all } \xi \in \mathbb R^{2d},
$$
for a constant $C>0$, and let $\mu^1\colon[0,T]\to \PP(\R^{2d})$ and $\mu^2\colon[0,T]\to \PP(\R^{2d})$ be continuous maps with respect to $\WW$ both satisfying 
\begin{equation}\label{hypo2+}
{\rm supp }\,\mu(t)\subset B(0,R)\quad\hbox{and}\quad {\rm supp }\,\nu(t)\subset B(0,R)
\end{equation}
for every $t \in [0, T]$, and $\mu_m^1$, $\mu_m^2$ two time-dependent atomic measures supported on the respective absolutely continuous trajectories   $t \mapsto (y_k^i(t),w_k^i(t))$, $i=1,2$ and $k =1,\dots,m$.
Consider  the flow maps ${\mathcal T}^{\mu^i,\mu_m^i}_t$, $i=1,2$,  associated to the systems
\begin{equation}\label{system+}
\begin{cases}
\dot X(t)=V(t)\\
\dot V(t)= H\star \mu^1(t) (X(t), V(t))+  H\star \mu_m^i(t) (X(t), V(t))
\end{cases}
\end{equation}
for $i=1,2$ respectively, on $[0, T]$. Fix $r>0$: then there exist a constant $\rho$ and a constant $L>0$, both depending only on $r$, $C$, $R$, and $T$ such that
\begin{equation}\label{contflow}
|{\mathcal T}^{\mu^1,\mu^1_m}_t(P_1)-{\mathcal T}^{\mu^2,\mu^2_m}_t(P_2)|\le e^{L t}|P_1-P_2|+ \int_0^t e^{L(s-t)}\|(H\star \mu^1(s)-H\star \mu^2(s)) + (H\star \mu^1_m(s)-H\star \mu^2_m(s))\|_{L^\infty(B(0,\rho))}\,ds
\end{equation}
whenever $|P_1|\le r$ and $|P_2|\le r$, for every $t \in [0, T]$.
\end{lemma}

\begin{proof}
Let $g_1$ and $g_2\colon [0,T]\times \R^{2d}\to \R^{2d}$  be the right-hand sides of \eqref{system}, and \eqref{system+}, respectively. As in  \eqref{trzu} we can find a constant $C'$ which depends only on $C$ and $R$ such that
\begin{equation}\label{trzu+}
 |H\star \mu^i(t) (P)| + |H\star \mu^i_m(t) (P)|\}\le C'(1+|P|) \quad i=1,2,
\end{equation}
for every $t \in [0, T]$ and every $P \in \R^{2d}$. Setting now $\hat C=1+C'$, it follows that $g_1$ and $g_2$ both satisfy \eqref{ttz} with $C$ replaced by $\hat C$. Therefore, for every $P_1$ and $P_2 \in \R^{2d}$ such that $|P_i|\le r$, $i=1,2$ and every $t\in [0,T]$, \eqref{gron} gives
$$
|{\mathcal T}^{\mu^i,\mu^i_m}_t(P_1)|\le \Big(r+ \hat CT \Big) \,e^{\hat C T}, \quad i=1,2.
$$
Set $\rho:=\Big(r+ \hat C T \Big) \,e^{\hat CT}$. Now, obviously
$$
\|g_1(t,\cdot)-g_2(t,\cdot)\|_{L^\infty(B(0,\rho))}=\|(H\star \mu^1(s)-H\star \mu^2(s)) + (H\star \mu^1_m(s)-H\star \mu^2_m(s))\|_{L^\infty(B(0,\rho))}
$$
for every $t \in [0, T]$.

Furthermore, by  \eqref{lipt} and the definition of $\rho$, the Lipschitz constant of $g_1(t, \cdot)$  on $B(0,\rho)$ can be estimated by a constant $L>0$ only depending on $R$, $C$, $r$ and $T$. With this, the conclusion follows at once from \eqref{gronvalla}.
\end{proof}
We additionally recall the following Lemmata (see, e.g., \cite[Lemma 3.11, Lemma 3,13, Lemma 3.15, and Lemma 4.7]{CanCarRos10} for their proofs).
\begin{lemma}\label{primstim}
Let $E_1$ and $E_2 \colon \R^n \to \R^n$ be two bounded Borel measurable functions. Then, for every $\mu \in \PP(\R^n)$ one has
\begin{equation*}
\WW((E_1)_\sharp \mu, (E_2)_\sharp \mu) \le \|E_1-E_2\|_{L^\infty({\rm supp}\,\mu)}\,.
\end{equation*}
If in addition $E_1$ is locally Lipschitz continuous, and $\mu$, $\nu \in \PP(\R^n)$ are both compactly supported on a ball $B_r$ of $\R^n$, then
\begin{equation}\label{transport}
\WW((E_1)_\sharp \mu, (E_1)_\sharp \nu) \le L_r \WW(\mu, \nu)\,,
\end{equation}
where $L_r$ is the Lipschitz constant of $E_1$ on $B_r$.
\end{lemma}

\begin{lemma}\label{secstim}
Let $H\colon \R^{2d} \to \R^{d}$ be a locally Lipschitz function satisfying \eqref{lingrowth}, let $\mu\colon[0,T]\to \PP(\R^{2d})$ and $\nu\colon[0,T]\to \PP(\R^{2d})$ be continuous maps with respect to $\WW$ both satisfying 
\begin{equation*}
{\rm supp }\,\mu(t)\subset B(0,R)\quad\hbox{and}\quad {\rm supp }\,\nu(t)\subset B(0,R)
\end{equation*}
for every $t \in [0, T]$. Then for every $\rho >0$ there exists a constant $L_{\varrho, R}$ such that
$$
\|H\star \mu(t)-H\star \nu(t)\|_{L^\infty(B(0,\rho)}\le L_{\varrho, R}\WW(\mu(t), \nu(t))
$$
for every $t \in [0, T]$.
\end{lemma}

\section*{Acknowledgement}

Massimo Fornasier acknowledges the support of the ERC-Starting Grant HDSPCONTR ``High-Dimensional Sparse Optimal Control''. Massimo Fornasier and Francesco Rossi  acknowledge the support of the DAAD-PHC (PROCOPE) Project 57049753 ``Sparse Control of Multiscale Models of Collective Motion''. Benedetto Piccoli and Francesco Rossi acknowledge for the support the NSFgrant \#1107444 (KI-Net).
\\
Massimo Fornasier and Benedetto Piccoli thank Peter A. Markowich for the stimulating discussions at the International Congress on Industrial and Applied Mathematics (ICIAM) 2011 in Vancouver, which led to the starting of this joint collaboration.

\bibliographystyle{abbrv}
\bibliography{biblioflock}

\end{document}